\documentclass[12pt,leqno]{article}
\usepackage{amsfonts}
\usepackage{amsmath, amsthm, amsfonts, amssymb, color}
\usepackage{mathrsfs}
\newtheorem{thm}{Theorem}[section]

\newtheorem{lem}[thm]{Lemma}

\theoremstyle{definition}

\newcommand{\scr}[1]{\mathscr #1}
\definecolor{wco}{rgb}{0.5,0.2,0.3}

\numberwithin{equation}{section} \theoremstyle{remark}
\newtheorem{rem}{Remark}[section]

\newcommand{\ua}{\uparrow}

\title{{\bf Long-Term Behaviors of Stochastic Interest Rate Models with Jumps and Memory}
}
\author{
{\bf  Jianhai Bao \, and \,  Chenggui Yuan}\\
 \footnotesize{Department of Mathematics,
Swansea University, Singleton Park, SA2 8PP, UK}\\
\footnotesize{majb@Swansea.ac.uk, C.Yuan@Swansea.ac.uk}}
\begin{document}
\def\R{\mathbb R}  \def\ff{\frac} \def\ss{\sqrt} \def\B{\mathbf
B}
\def\N{\mathbb N} \def\kk{\kappa} \def\m{{\bf m}}
\def\dd{\delta} \def\DD{\Delta} \def\vv{\varepsilon} \def\rr{\rho}
\def\<{\langle} \def\>{\rangle} \def\GG{\Gamma} \def\gg{\gamma}
  \def\nn{\nabla} \def\pp{\partial} \def\EE{\scr E}
\def\d{\text{\rm{d}}} \def\bb{\beta} \def\aa{\alpha} \def\D{\scr D}
  \def\si{\sigma} \def\ess{\text{\rm{ess}}}
\def\beg{\begin} \def\beq{\begin{equation}}  \def\F{\scr F}
\def\Ric{\text{\rm{Ric}}} \def\Hess{\text{\rm{Hess}}}
\def\e{\text{\rm{e}}} \def\ua{\underline a} \def\OO{\Omega}  \def\oo{\omega}
 \def\tt{\tilde} \def\Ric{\text{\rm{Ric}}}
\def\cut{\text{\rm{cut}}} \def\P{\mathbb P} \def\ifn{I_n(f^{\bigotimes n})}
\def\C{\scr C}      \def\aaa{\mathbf{r}}     \def\r{r}
\def\gap{\text{\rm{gap}}} \def\prr{\pi_{{\bf m},\varrho}}  \def\r{\mathbf r}
\def\Z{\mathbb Z} \def\vrr{\varrho} \def\ll{\lambda}
\def\L{\scr L}\def\Tt{\tt} \def\TT{\tt}\def\II{\mathbb I}
\def\i{{\rm in}}\def\Sect{{\rm Sect}}\def\E{\mathbb E} \def\H{\mathbb H}
\def\M{\scr M}\def\Q{\mathbb Q} \def\texto{\text{o}} \def\LL{\Lambda}
\def\Rank{{\rm Rank}} \def\B{\scr B} \def\i{{\rm i}} \def\HR{\hat{\R}^d}
\def\to{\rightarrow}\def\l{\ell}
\def\8{\infty}

\date{}

\maketitle

\begin{abstract}
In this paper we show the convergence of the long-term
return $t^{-\mu}\int_0^tX(s)\d s$ for some $\mu\geq1$, where $X$ is
the short-term interest rate which follows an extension of
Cox-Ingersoll-Ross type model with jumps and memory, and, as an
application, we also investigate the corresponding behavior of
two-factor Cox-Ingersoll-Ross model with jumps and memory.

 \end{abstract}

\section{Introduction}
Cox, Ingersoll and Ross \cite{cir85} propose  the short-term rate
dynamics as
\begin{equation*}
\d S(t)=\kappa(\gamma-S(t))\d t+\sigma\sqrt{S(t)}\d W(t),
\end{equation*}
where $\kappa,\gamma$ and $\sigma$ are positive constants. This
model is also named mean-reverting square root process or
Cox-Ingersoll-Ross (CIR) model. In order to better capture the
properties of the empirical data, there are many extensions of the
CIR model, e.g., Chan, Karolyi, Longstaff and Sander \cite{chan92}
generalize  the CIR model as
\begin{equation*}
\d S(t)=\kappa(\gamma-S(t))\d t+\sigma S(t)^\theta\d W(t),
\end{equation*}
where $\theta\ge 1/2.$  Another generalization of the CIR model is
to use the {\it regime-switching} such as in Ang and Bekaert
\cite{ang02} and Gary \cite{gray96}, to name a few.  On the other
hand, taking into consideration the influence of past events, many
scholars introduce  {\it delay }
 to the financial models. For example,  in his paper \cite{ben04}, Benhabin considers a linear, flexible price
model, where  nominal interest rates are measured by a flexible
distributed delay. In their paper \cite{ahp07}, Arriojas, Hu and
Mohammed take the delay into the consideration for  the price
process of underlying assets and develop the Black-Scholes formula.
Moreover, {\it jump processes} are also used in the financial
models, e.g.,  \cite{Bar, 4,   Hen, Merc}, and the references
therein.

There are extensive literature on quantitative and qualitative
properties of the generalized CIR-type models.   Different
convergence results and corresponding applications of the long-term
return can be found in \cite{dd95,dd98,z09};  Strong convergence of
the Monte Carlo simulations are studied in
\cite{dd98b,gr11,wmc08,wmc09}, and the representations  of solutions
are presented in \cite{ahp07,c92}. We here would like to point out
that Deelstra and Delbaen \cite{dd95,dd98} investigate the
long-term returns of the CIR model. Zhao \cite{z09} extends the results of  \cite{dd95,dd98} to the jump model. In the present  paper we will consider the effect of the past and jump in the determination of the interest model and study the long-term return of the
 stochastic interest rate model with jumps and memory.

In the following section, we will introduce the mathematical model
and notation, the long-term return will be studied in section 3, and
an application of the main result, Theorem \ref{Theorem 3.1},  is
discussed in the last section.

\section{Preliminaries}
Throughout this paper, let $(\Omega,
  {\cal{F}},\{{\cal{F}}_t\}_{t\ge 0}, \mathbb{P})$ be a complete probability space with a filtration
  $\{{\cal{F}}_t\}_{t\ge 0}$ satisfying
 the usual conditions (i.e. it is right continuous and ${\cal{F}}_{0}$ contains all
 $\mathbb{P}$-null sets). Let $W(t)$ be a scalar Brownian motion. Let  $\mathcal
{B}(\mathbb{R}_+)$ be the Borel $\sigma$-algebra on $\mathbb{R}_+$,
and $\lambda(dx)$ a $\sigma$-finite measure defined on $\mathcal
{B}(\mathbb{R}_+)$. Let $p=(p(t)),t\in D_p$, be a stationary
$\mathcal {F}_t$-Poisson point process on $\mathbb{R}_+$ with
characteristic measure $\lambda(\cdot)$. Denote by $N(dt,du)$ the
Poisson counting measure associated with $p$, i.e.,
$N(t,U)=\sum_{s\in D_p, s\leq t}I_{U}(p(s))$ for $U\in\mathcal
{B}(\mathbb{R}_+)$. We assume $\lambda(U)<\infty$ and let
$\tilde{N}(dt,du):=N(dt,du)-dt\lambda(du)$ be the compensated
Poisson measure associated with $N(dt,du)$.  For
the sake of convenience,  we will denote  $C>0$  a generic constant
whose values may change from lines to lines.

Consider stochastic interest rate
model with jumps and memory,
\begin{equation}\label{eq1}
\begin{cases}
\d X(t)=\{2\beta X(t) +\delta(t)\}\d t+\sigma
X^{\gamma}(t-\tau)\sqrt{|X(t)|}\d W(t)\\
\ \ \ \ \ \ \ \ \ \ \ +\int_Ug(X(t-),u)\tilde{N}(\d t,\d u),\\
 X_0=\xi\in\mathscr{C},
\end{cases}
\end{equation}
where $X(t-)=\lim_{s\uparrow t}X(s).$
We make the following assumptions:
\begin{enumerate}
\item[\textmd{(A1)}] $\beta<0$, $\sigma>0$ and
$\gamma\in[0,\frac{1}{2}).$
\item[\textmd{(A2)}]
$\delta:\Omega\times\mathbb{R}_+\rightarrow\mathbb{R}_+$, and there
exist constants $\mu\geq1$ and $\nu\geq0$ such that
\begin{equation*}
\lim_{t\rightarrow\infty}\frac{1}{t^\mu}\int_0^t\delta(s)\d s:=\nu \
\ \ \mbox{ a.s. }
\end{equation*}
\item[\textmd{(A3)}] $g:\Omega\times\mathbb{R}\rightarrow\mathbb{R}$ with $g(0,u)=0$ and there exists $K>0$ such that
\begin{equation*}
\int_U|g(x,u)-g(y,u)|^2\lambda(\d u)\leq K|x-y|^2
\end{equation*}
for arbitrary $x,y\in\mathbb{R}$.
\item[\textmd{(A4)}] For any $\theta\in[0,1]$, $x+\theta
g(x,u)\geq0$ whenever $x>0$.
\end{enumerate}

 Compared with the existing
literature, our key contributions of this paper are as follows:
\begin{itemize}
\item We investigate the almost sure convergence of the
long-term return $t^{-\mu}\int_0^tX(s)\d s$ for some $\mu\geq1$, and
extend the results of Deelstra and Delbaen \cite{dd95,dd98} and Zhao
\cite{z09}.  Since the  jumps and memory are involved, we will see
the generalization is not trivial.
\item As an application, we also study the long-term behaviors for a
class of two-factor CIR models with jumps and memory, where we extend the result of  \cite[Theorem 2]{z09}.
\end{itemize}

\section{Almost Sure Convergence of Long-Term Returns}

For our purposes we first prepare or recall several auxiliary
lemmas.
\begin{lem}\label{lemma 1.1}
{\rm Under $(A1)$-$(A4)$, Eq. \eqref{eq1} admits a unique
nonnegative solution $(X(t))_{t\geq0}$ for any $\xi\in\mathscr{C}$.
}
\end{lem}
\begin{proof}
Application of \cite[Theorem 2.1 \& 2.2]{z10} gives that Eq.
\eqref{eq1} has a unique strong solution $X(t)$ on $[0,\tau]$.
Repeating this procedure we see that Eq. \eqref{eq1} also admits a
unique strong solution $X(t)$ on $[\tau,2\tau]$. Hence Eq.
\eqref{eq1} has a unique strong solution $X(t)$ on the horizon $t
\ge 0$. Moreover, carrying out a similar argument to that of
\cite[Theorem 2.1]{wmc09}, we can deduce that there exists $C>0$
such that for any $q>0$
\begin{equation}\label{eq27}
\mathbb{E}|X(t)|^q\leq C, \ \ t\in[0,T].
\end{equation}
To end the proof, it is sufficient to show the
nonnegative property of the solution $(X(t))_{t\in[0,T]}$ for any $T>0$. We
adopt the method of Yamada and Watanabe \cite{yw71}. Let
$a_0=1$ and $a_k=\exp(-k(k+1)/2),k=1,2\cdots$. Then it is easy to
see that $ \int_{a_k}^{a_{k-1}}\frac{1}{kx}\d x=1 $ and consequently
there is a continuous nonnegative function
$\psi_k(x),x\in\mathbb{R}_+$, which possesses the support
$(a_k,a_{k-1})$, has integral $1$ and satisfies
$\psi_k(x)\leq\frac{2}{kx}$. Define an auxiliary function
$\phi_k(x)=0$ for $x\geq0$ and
$$\phi_k(x):=\int_0^{-x}\d y\int_0^y\psi_k(u)\d u, \ \ x<0.$$ By a straightforward
computation, $\phi_k\in C^2(\mathbb{R};\mathbb{R}_+)$ has the
following properties:
\begin{enumerate}
\item[\textmd{(i)}] $-1\leq\phi_k'(x)\leq0$ for $-a_{k-1}<x<-a_k$, or
otherwise $\phi_k'(x)=0$;
\item[\textmd{(ii)}] $|\phi_k''(x)|\leq\frac{2}{k|x|}$ for $-a_{k-1}<x<-a_k$, or
otherwise $\phi_k''(x)=0$;
\item[\textmd{(iii)}] $x^--a_{k-1}\leq\phi_k(x)\leq x^-,
\ \ x\in\mathbb{R}$.
\end{enumerate}
An application of the It\^o formula yields that for any $t\in(0,T]$
\begin{equation*}
\begin{split}
&\mathbb{E}\phi_k(X(t))=\mathbb{E}\phi_k(\xi(0))+\mathbb{E}\int_0^t\phi'_k(X(s))\{2\beta
X(s) +\delta(s)\}\d s\\
&\quad+\frac{\sigma^2}{2}\mathbb{E}\int_0^t\phi''_k(X(s))X^{2\gamma}(s-\tau)X(s)\d
s\\
&\quad+\mathbb{E}\int_0^t\int_U\{\phi_k(X(s)+g(X(s),u))-\phi_k(X(s))-\phi'_k(X(s))g(X(s),u)\}\lambda(\d
u)\d s.
\end{split}
\end{equation*}
By the properties (i)-(iii), Taylor's expansion and $(A4)$, it then
follows from \eqref{eq27} that
\begin{equation*}
\begin{split}
&\mathbb{E}\phi_k(X(t))
\leq\frac{C}{k}+\mathbb{E}\int_0^t\int_U\int_0^1\{(\phi'_k(\theta
g(X(s),u)+X(s))-\phi'_k(X(s)))g(X(s),u)\}\d \theta \lambda(\d u)\d
s\\
&=\frac{C}{k}+\mathbb{E}\int_0^t\int_U\int_0^1\{(\phi'_k(\theta
g(X(s),u)+X(s))-\phi'_k(X(s)))g(X(s),u)\}{\bf 1}_{\{X(s)>0\}}\d
\theta \lambda(\d u)\d
s\\
&\quad+\mathbb{E}\int_0^t\int_U\int_0^1\{(\phi'_k(\theta
g(X(s),u)+X(s))-\phi'_k(X(s)))g(X(s),u)\}{\bf 1}_{\{X(s)\leq0\}}\d
\theta \lambda(\d u)\d s\\
&\leq\frac{C}{k}+2K^{\frac{1}{2}}\lambda^{\frac{1}{2}}(U)\mathbb{E}\int_0^tX^-(s){\bf
1}_{\{X(s)\leq0\}}\d s\\
&\leq\frac{C}{k}+2K^{\frac{1}{2}}\lambda^{\frac{1}{2}}(U)\mathbb{E}\int_0^t\{a_{k-1}+\phi_k(X(s))\}\d
s\\
&=\frac{C}{k}+2K^{\frac{1}{2}}\lambda^{\frac{1}{2}}(U)Ta_{k-1}+2K^{\frac{1}{2}}
\lambda^{\frac{1}{2}}(U)\int_0^t\mathbb{E}\phi_k(X(s))\d s, \ \ \ \
\ t\in(0,T].
\end{split}
\end{equation*}
This, together with the Gronwall inequality, gives that
\begin{equation*}
\mathbb{E}X^-(t)-a_{k-1}\leq\mathbb{E}\phi_k(X(t))\leq
C\Big(\frac{1}{k}+a_{k-1}\Big), \ \ \ \ t\in(0,T].
\end{equation*}
Thus, $\mathbb{E}X^-(t)=0$ as $k\rightarrow\infty$ and therefore
$X(t)\geq0$ a.s. for any $t\in(0,T]$. Hence the nonnegative property
of the solution $(X(t))_{t\geq0}$ follows from the arbitrariness of
$T>0$.
\end{proof}

\begin{rem}
{\rm There are some examples such that $(A4)$ holds, e.g., for
$x\in\mathbb{R}$ and $u\in U$, $g(x,u)\geq0$ or $-g(x,u)\leq x$
whenever $g(x,u)\leq0$. }
\end{rem}

\begin{rem}
{\rm Wu, Mao and Chen \cite{wmc08} study the strong convergence of
Monte Carlo simulations of the mean-reverting square root process
with jump
\begin{equation}\label{eq32}
\d S(t)=\alpha[\mu- S(t)]\d t+ \sigma\sqrt{|S(t)|}\d W(t)+\delta
S(t-)\d \tilde{N}(t),
\end{equation}
where $\alpha,\mu,\sigma>0$, and, in particular, investigate the
nonnegative property of $S(t)$. It is easy to see that our model is
a generalization of  model \eqref{eq32}. Zhao \cite{z09} also showed
the nonnegative property of \eqref{eq1} with $\gamma=0$, $g(x,u)=0$
for $x<0$ and $\int_Ug^2(x,u)\lambda(\d u)\leq K|x|$ for some
constant $K>0$. Moreover, we would like to point that our goal is to
study the long-term return, which is different from those of
\cite{wmc08,wmc09}.}
\end{rem}

\begin{lem}\label{Lemma 2.2}
{\rm Let $(A1)$-$(A4)$ hold and assume further that $4\beta+K<0$.
Then there exist $\kappa>0$ and $C>0$ such that
\begin{equation}\label{eq7}
\begin{split}
\mathbb{E}(e^{-\kappa\beta \rho}X^2(\rho)) &\leq C+C\mathbb{E}\int_0^\rho
e^{-\kappa\beta s}(\delta^2(s)+1)\d s,
\end{split}
\end{equation}
where $\rho>0$ is a bounded stopping time.}
\end{lem}
\begin{proof}
 We first recall the Young inequality: for any $a,b>0$ and
$\alpha\in(0,1)$
\begin{equation}\label{eq14}
a^\alpha b^{1-\alpha}\leq \alpha a+(1-\alpha)b.
\end{equation}
Let $\kappa>0$ and $\epsilon>0$ be arbitrary. By the It\^o formula,
$(A3)$ and the Young inequality \eqref{eq14}, we obtain that
\begin{equation*}
\begin{split}
\d (e^{-\kappa\beta t}X^2(t))&=-\kappa\beta e^{-\kappa\beta t}X^2(t)\d t+e^{-\kappa\beta
t}\d X^2(t)\\
&=e^{-\kappa\beta t}\Big\{(4-\kappa)\beta X^2(t)+\sigma^2X(t)X^{2\gamma}(t-\tau)+2\delta(t)X(t)\\
&\quad+\int_Ug^2(X(t),u)\lambda(\d u)\Big\}\d t+M_1(t)+M_2(t)\\
&\leq e^{-\kappa\beta t}\{((4-\kappa)\beta+\epsilon+K)  X^2(t)+C_1(\epsilon)X^{4\gamma}(t-\tau)+C_1(\epsilon)\delta^2(t)\}\d t\\
&\quad+M_1(t)+M_2(t)\\
 &\leq e^{-\kappa\beta t}\{((4-\kappa)\beta+\epsilon+K)  X^2(t)+\epsilon e^{\kappa\beta\tau} X^2(t-\tau)+C_1(\epsilon)\delta^2(t)+C_2(\epsilon)\}\d t\\
&\quad+M_1(t)+M_2(t)
\end{split}
\end{equation*}
for some  constants $C_1(\epsilon)>0$ and $C_2(\epsilon)>0$,
dependent on $\epsilon$, where $ M_1(t):=2\sigma e^{-\kappa\beta t}
X^{\frac{3}{2}}(t)X^{\gamma}(t-\tau)\d W(t)$ and $
M_2(t):=e^{-\kappa\beta
t}\int_U\{g^2(X(t),u)+2X(t)g(X(t),u)\}\tilde{N}(\d t,\d u). $
Integrating from $0$ to $\rho$ and taking expectations  on both
sides, we arrive at
\begin{equation*}
\begin{split}
\mathbb{E}(e^{-\kappa\beta \rho}X^2(\rho))&\leq
C\|\xi\|^2+((4-\kappa)\beta+2\epsilon+K)\mathbb{E}\int_0^\rho e^{-\kappa\beta
s}X^2(s)\d
s\\
&\quad+(C_1(\epsilon)\vee
C_2(\epsilon))\mathbb{E}\int_0^\rho(\delta^2(s)+1)\ ds.
\end{split}
\end{equation*}
Due to $4\beta+K<0$, we can choose $\kappa>0$ and $\epsilon>0$ such that
$(4-\kappa)\beta+2\epsilon+K=0$, and therefore \eqref{eq7} follows
immediately.
\end{proof}

For the future use, we cite the following as a lemma.
\begin{lem}\label{Lemma 1.2}
{\rm (\cite[Kronecker's lemma, p164]{dd95}) Assume that $Y(t)$ is a
c\`{a}dl\`{a}g semimartingale and that $f(t)$ is a strictly positive
increasing function with $f(t)\rightarrow\infty$ as
$t\rightarrow\infty$. If $\int_0^\infty\frac{\d Y(t)}{f(t)}$ exists
a.s., then $\frac{Y(t)}{f(t)}\rightarrow0$ a.s. }
\end{lem}

We now state our main result.
\begin{thm}\label{Theorem 1.3}
{\rm Let $(A1)$-$(A4)$ hold and $4\beta+K<0$. Assume further that
there exist $\lambda>0$ and $\theta\in[1,2\mu]$ such that
\begin{equation}\label{eq2}
\limsup_{t\rightarrow\infty}\frac{1}{t^\theta}\int_0^t\delta^2(s)ds\leq
\lambda  \ \ \ \mbox{ a.s. }
\end{equation}
Then
\begin{equation}\label{eq3}
\lim\limits_{t\rightarrow\infty}\frac{1}{t^\mu}\int_0^t\Big\{X(s)+\frac{\delta(s)}{2\beta}\Big\}\d
s=0 \ \ \mbox{ a.s. }
\end{equation}
}
\end{thm}
\begin{proof}
It is easy to see from Eq. \eqref{eq1} that
\begin{equation}\label{eq4}
\begin{split}
\int_0^t\Big\{X(s)+\frac{\delta(s)}{2\beta}\Big\}\d
s&=\frac{X(t)-\xi(0)}{2\beta}-\frac{\sigma}{2\beta}\int_0^tX^{\gamma}(s-\tau)\sqrt{|X(s)|}\d
W(s)\\
&\quad-\frac{1}{2\beta}\int_0^t\int_Ug(X(s-),u)\tilde{N}(\d s,\d u).
\end{split}
\end{equation}
On the other hand, application of  It\^o's formula to $e^{-2\beta t}X(t)$
yields that
\begin{equation*}
\begin{split}
X(t)&=e^{2\beta t}\Big\{\xi(0)+\int_0^te^{-2\beta s}\delta(s)\d
s+\sigma\int_0^te^{-2\beta s}X^{\gamma}(s-\tau)\sqrt{|X(s)|}\d
W(s)\\
&\quad+\int_0^t\int_Ue^{-2\beta s}g(X(s-),u)\tilde{N}(\d s,\d
u)\Big\}.
\end{split}
\end{equation*}
Thus, substituting this into \eqref{eq4} one has
\begin{equation*}
\begin{split}
&\frac{1}{t^\mu}\int_0^t\Big\{X(s)+\frac{\delta(s)}{2\beta}\Big\}\d
s=\frac{(e^{2\beta
t}-1)\xi(0)}{2\beta t^\mu}\\
&\quad+\frac{1}{2\beta t^\mu}\int_0^te^{2\beta(t-
s)}\delta(s)\d s\\
&\quad+\frac{\sigma}{2\beta}\Big(1+\frac{1}{t}\Big)^\mu\frac{1}{e^{-2\beta t}(1+t)^\mu}\int_0^te^{-2\beta s}\delta(s)X^{\gamma}(s-\tau)\sqrt{|X(s)|}\d W(s)\\
&\quad-\frac{\sigma}{2\beta
}\Big(1+\frac{1}{t}\Big)^\mu\frac{1}{(1+t)^\mu}\int_0^tX^{\gamma}(s-\tau)\sqrt{|X(s)|}\d
W(s)\\
&\quad-\frac{1}{2\beta}\Big(1+\frac{1}{t}\Big)^\mu\frac{1}{(1+t)^\mu}\int_0^t\int_Ug(X(s-),u)\tilde{N}(\d s,\d u)\\
&\quad+\frac{1}{2\beta}\Big(1+\frac{1}{t}\Big)^\mu\frac{1}{e^{-2\beta
t}(1+t)^\mu}\int_0^t\int_Ue^{-2\beta s}g(X(s-),u)\tilde{N}(\d s,\d
u)\\
 &:=I_1(t)+I_2(t)+\frac{1}{2\beta}\Big(1+\frac{1}{t}\Big)^\mu
I_3(t)-\frac{\sigma}{2\beta }\Big(1+\frac{1}{t}\Big)^\mu
I_4(t)\\
&\quad-\frac{1}{2\beta}\Big(1+\frac{1}{t}\Big)^\mu
I_5(t)+\frac{1}{2\beta}\Big(1+\frac{1}{t}\Big)^\mu I_6(t).
\end{split}
\end{equation*}
To derive the desired assertion \eqref{eq3}, it is sufficient to
verify that $I_i(t)\rightarrow0$ a.s., $i=1,\cdots,6$, as
$t\rightarrow\infty$ respectively. Due to $\beta<0$ and $\mu\geq1$,
it is trivial that $I_1(t)\rightarrow0$ as $t\rightarrow\infty$.
Following a similar argument to that of \cite[p168]{dd95} and noting
that
$\lim_{t\rightarrow\infty}[(1+t)^\mu-(1+t-\sqrt{t})^\mu]/(1+t)^\mu=0$,
by $(A2)$ we can also deduce that $I_2(t)\rightarrow0$ a.s. for
$t\rightarrow\infty$. Next, in order to show $I_3(t)\rightarrow0$
a.s. and $I_4(t)\rightarrow0$ a.s. whenever $t\rightarrow\infty$,
respectively, by Lemma \ref{Lemma 1.2}  it suffices to check that
\begin{equation}\label{eq5}
\int_0^\infty \frac{X^{\gamma}(t-\tau)\sqrt{|X(t)|}}{(1+t)^\mu}\d
W(t) \ \ \mbox{ exists a.s. }
\end{equation}
For each $n>\|\xi\|$ define a stopping time
\begin{equation*}
\tau_n:=\inf\Big\{t\geq0\Big|\int_0^t\frac{\delta^2(s)}{(1+s)^{2\mu}}\d
s\geq n\Big\}.
\end{equation*}
In the light of  \eqref{eq2} there exists an $L>0$ such that
\begin{equation}
\int_0^t\delta^2(s)ds\leq L(1+t)^\theta \ \ \mbox{ a.s. }
\end{equation}
This, together with $\theta\in[1,2\mu]$, leads to
\begin{equation*}
\begin{split}
\int_0^\infty\frac{\delta^2(s)}{(1+s)^{2\mu}}\d
s&=\lim_{s\rightarrow\infty} \frac{\int_0^s\delta^2(u)\d
u}{(1+s)^{2\mu}}+2\mu\int_0^\infty\Big(\int_0^s\delta^2(u)\d
u\Big)\frac{\d
s}{(1+s)^{2\mu+1}}\\
&\leq\lim_{s\rightarrow\infty}
\frac{L}{(1+s)^{2\mu-\theta}}+2L\int_0^\infty\frac{1}{(1+s)^{2\mu+1-\theta}}\d
s\\
&<\infty \ \ \ \mbox{ a.s. }
\end{split}
\end{equation*}
Hence $\{\tau_n=\infty\}\uparrow\Omega$ and consequently, it is
sufficient to verify \eqref{eq5}  on $\{\tau_n=\infty\}$.
Furthermore, observing that
\begin{equation*}
J(t):=\int_0^t\frac{X^{\gamma}(s-\tau)\sqrt{|X(s)|}}{(1+s)^\mu}1_{\{s\leq\tau_n\}}\d
W(s)
\end{equation*}
is a local martingale, we only need to check that $J(t)$ is  an
$L^2$-bounded martingale. By the It\^o isometry and the Young
inequality \eqref{eq14} we can
obtain that
\begin{equation*}
\begin{split}
\mathbb{E}\left|J(t)\right|^2&=\int_0^{t}
\frac{\mathbb{E}\{X^{2\gamma}(s-\tau)X(s)\}1_{\{s\leq\tau_n\}}}{(1+s)^{2\mu}}\d
s\\
&\leq \int_0^{t}
\frac{\mathbb{E}\{X^2(s)1_{\{s\leq\tau_n\}}\}}{2(1+s)^{2\mu}}\d
s+\int_0^{t}
\frac{\mathbb{E}\{X^{4\gamma}(s-\tau)1_{\{s\leq\tau_n\}}\}}{2(1+s)^{2\mu}}\d
s\\
&\leq \int_0^{t} \frac{(1-2\gamma)}{2(1+s)^{2\mu}}\d s+\int_0^{t}
\frac{\mathbb{E}\{X^2(s)1_{\{s\leq\tau_n\}}\}}{2(1+s)^{2\mu}}\d
s+\gamma\int_0^{t}
\frac{\mathbb{E}\{X^2(s-\tau)1_{\{s\leq\tau_n\}}\}}{(1+s)^{2\mu}}\d
s\\
&\leq\frac{1-2\gamma}{2(2\mu-1)}+\int_0^{t} \frac{e^{-\kappa\beta
s}\mathbb{E}\{e^{\kappa\beta
(s\wedge\tau_n)}X^2(s\wedge\tau_n)\}}{2(1+s)^{2\mu}}\d s\\
&\quad+\gamma\int_0^{t}
\frac{e^{-\kappa\beta(s-\tau)}\mathbb{E}\{e^{\kappa\beta(s\wedge\tau_n-\tau)}X^2(s\wedge\tau_n-\tau)\}}{(1+s)^{2\mu}}\d
s\\
&:=(1-2\gamma)/(4\mu-2)+J_1(t)+J_2(t).
\end{split}
\end{equation*}
 By  \eqref{eq7} with $\kappa>0$ it follows that
\begin{equation}\label{eq6}
\begin{split}
J_1(t)&\leq C\int_0^{t} \frac{\Big\{1+s+e^{\kappa\beta
s}\mathbb{E}\int_0^{s\wedge\tau_n} e^{-\kappa\beta
r}\delta^2(r)dr\Big\}}{2(1+s)^{2\mu}}\d s\\
&\leq C\int_0^{t} \frac{\Big\{1+s+e^{\kappa\beta s}\int_0^se^{-\kappa\beta
r}\mathbb{E}(\delta^2(r)1_{\{r\leq\tau_n\}})\d r\Big\}}{2(1+s)^{2\mu}}\d s\\
&\leq C+C\int_0^{t}\frac{e^{\kappa\beta
s}}{2(1+s)^{2\mu}}\int_0^se^{-\kappa\beta
r}\mathbb{E}(\delta^2(r)1_{\{r\leq\tau_n\}})\d r\d s\\
&=C+C\int_0^{t}e^{\kappa\beta
r}\mathbb{E}(\delta^2(r)1_{\{r\leq\tau_n\}})\int_r^t\frac{e^{-\kappa\beta
s}}{2(1+s)^{2\mu}}\d s\d r\\
&\leq C+C\mathbb{E}\int_0^{\tau_n}\frac{\delta^2(r)}{(1+r)^{2\mu}}\d
r\\
&\leq C\Big(1+n\Big).
\end{split}
\end{equation}
Noting that
\begin{equation*}
\begin{split}
J_2(t)&\leq C+\gamma\int_\tau^t
\frac{e^{\kappa\beta(s-\tau)}(\|\xi\|^2+\mathbb{E}\{e^{-\kappa\beta(s\wedge\tau_n-\tau)}X^2(s\wedge\tau_n-\tau)1_{\{s\wedge\tau_n>\tau\}}\})}{2(1+s)^{2\mu}}\d
s, \ \ t>\tau,
\end{split}
\end{equation*}
and carrying out the similar argument to that of \eqref{eq6}, we can
 conclude that there exists $C(n,\mu,\alpha)>0$ such that
$J_2(t)\leq C(n,\mu,\alpha)$. Finally, $I_5(t)\rightarrow0$ a.s. and
$I_6(t)\rightarrow0$ a.s. follows whenever $t\rightarrow\infty$ by
observing
\begin{equation*}
\mathbb{E}\Big(\int_0^t\int_U\frac{g(X(s-),u)}{(1+s)^\mu}1_{\{s\leq\tau_n\}}\tilde{N}(\d
s,\d
u)\Big)^2=\mathbb{E}\int_0^t\int_U\frac{g^2(X(s-),u)}{(1+s)^{2\mu}}1_{\{s\leq\tau_n\}}\lambda(\d
u)\d s,
\end{equation*}
and following the previous argument, and the proof is therefore
complete.
\end{proof}

\begin{rem}
{\rm For $\delta(t)=t^{\mu-1},t\geq0$, and $\theta=2\mu-1$, it is
trivial to see that both $(A2)$ and \eqref{eq2} are true. }
\end{rem}

\begin{rem}
{\rm For $\beta<0,\sigma>0$ and $\gamma\in[0,\frac{1}{2})$, Theorem
\ref{Theorem 1.3} clearly applies to the generalized mean-reverting
model
\begin{equation*}
\begin{cases}
\d X(t)=\{2\beta X(t) +\delta(t)\}\d t+\sigma
|X(t)|^{\frac{1}{2}+\gamma}\d W(t),\\
X(0)=x>0,
\end{cases}
\end{equation*}
where Deelstra and Delbaen \cite{dd95} investigated the long-term
returns of such model with $\gamma=0$. Moreover,  Zhao \cite{z09}
discussed the long-time behavior of the stochastic interest rate
model \eqref{eq1} with $\gamma=0,g(x,u)=0$ for $x<0,u\in U$, and
\begin{equation}\label{eq28}
\int_Ug^2(x,u)\lambda(\d u)\leq K|x| \ \ \mbox{ for some constant }
K>0.
\end{equation}
Clearly, the linear case $g(x,u)=C|u|x$ for some $C>0$ does not
satisfy \eqref{eq28}, however, Theorem \ref{Theorem 1.3} is
available for such fundamental case. }
\end{rem}

\section{An Application to Two-Factor CIR Model}
In this section we turn to an application of Theorem \ref{Theorem
1.3}. Let  $W_1(t),W_2(t)$ be Brownian motions, and $N_1(\d t,\d
u),N_2(\d t,\d u)$ Poisson counting measures with characteristic
measures $\lambda_1(\cdot)$ and $\lambda_2(\cdot)$ respectively,
defined on $(\Omega,\mathcal {F},\mathbb{P},\{\mathcal
{F}_t\}_{t\geq0})$. Consider the following two-factor model with
jumps and  memory
\begin{equation}\label{eq13}
\begin{cases}
\d X(t)=\{2\beta_1 X(t) +\delta(t)\}\d t+\sigma_1
X^{\gamma_1}(t-\tau)\sqrt{|X(t)|}\d W_1(t)\\
\ \ \ \ \ \ \ \ \ \ \quad+\vartheta_1 X(t)\int_Uu\tilde{N}_1(\d t,\d u),\\
\d Y(t)=\{2\beta_2 Y(t) +X(t)\}\d t+\sigma_2
Y^{\gamma_2}(t-\tau)\sqrt{|Y(t)|}\d W_2(t)\\
\ \ \ \ \ \ \ \ \ \ \quad+\vartheta_2Y(t)\int_Uu\tilde{N}_2(\d t,\d
u)
\end{cases}
\end{equation}
with initial data $(X(t),Y(t))=(\xi(t),\eta(t)),t\in[-\tau,0]$,
where  $\xi,\eta\in\mathscr{C}$.

We assume that
\begin{enumerate}
\item[\textmd{(A5)}] $\beta_1<0,\sigma_1>0$ and $\gamma_1\in[0,\frac{1}{2}),\vartheta_1>0$, $\delta(t)$ satisfies
$(A2)$;
\item[\textmd{(A6)}]
$\beta_2<0,\sigma_2>0$, $\gamma_2\in[0,\frac{1}{2})$,
$\vartheta_2>0$ and $ \vartheta_2^2\int_Uu^2\lambda_2(\d
u)<-4\beta_2;$
\item[\textmd{(A7)}] For $\theta\in[1,2\mu]$ (where $\mu$ is defined in $(A2)$), $\int_0^\infty\frac{\delta^4(t)}{(1+t)^{2\theta}}\d
t<\infty$ a.s.
\end{enumerate}

\begin{lem}
{\rm Let $(A5)$ and $(A6)$ hold and assume that
\begin{equation}\label{eq30}
\vartheta_1^2\int_Uu^2(6+4\vartheta_1u+\vartheta_1^2u^2)\lambda_1(\d
u)=:\Gamma(\vartheta_1,\lambda_1)<-8\beta_1.
\end{equation}
Then \eqref{eq13} admits a unique nonnegative solution
$(X(t),Y(t))_{t\geq0}$, and  there exist $\kappa>0$ and $C>0$ such
that
\begin{equation}\label{eq29}
\begin{split}
\mathbb{E}(e^{-\kappa\beta_1 \rho}X^4(\rho)) &\leq
C+C\mathbb{E}\int_0^\rho e^{-\kappa\beta_1 s}(\delta^4(s)+1)\d s,
\end{split}
\end{equation}
where $\rho>0$ is a bounded stopping time.}
\end{lem}

\begin{proof}
By Lemma \ref{lemma 1.1}, under $(A5)$ and $(A6)$, \eqref{eq13}
admits a unique nonnegative solution $(X(t),Y(t))_{t\geq0}$. By the
It\^o formula and the Young inequality \eqref{eq14}, compute
\begin{equation}\label{eq31}
\begin{split}
\d(e^{-\kappa\beta_1 t}X^4(t))&=-\kappa\beta_1 e^{-\kappa\beta_1
t}X^4(t)\d t+e^{-\kappa\beta_1
t}\d X^4(t)\\
&=e^{-\kappa\beta_1 t}\Big\{(8-\kappa)\beta_1X^4(t)+4\delta(t)X^3(t)+6\sigma_1^2X^3(t)X^{2\gamma_1}(t-\tau)\\
&\quad+\int_U((1+\vartheta_1u)^4-1-4\vartheta_1u)\lambda_1(\d u)X^4(t)\Big\}+\tilde{M}_1(t)+\tilde{M}_2(t)\\
 &\leq e^{-\kappa\beta_1
t}\{((8-\kappa)\beta_1+\epsilon+\Gamma(\vartheta_1,\lambda_1))
X^4(t)+\epsilon e^{\kappa\beta_1\tau}X^4(t-\tau)\\
&\quad+C(\epsilon)(\delta^4(t)+1)\}\d
t+\tilde{M}_1(t)+\tilde{M}_2(t)
\end{split}
\end{equation}
for any $\kappa>0$ and sufficiently small $\epsilon>0$, where
$\tilde{M}_1(t)$ and $\tilde{M}_2(t)$ are two local martingales.
Then \eqref{eq29} can be obtained by integrating from $0$ to $\rho$,
taking expectations on both sides of \eqref{eq31} and, in
particular,  choosing $\kappa>0$ and $\epsilon>0$ such that
$(8-\kappa)\beta_1+2\epsilon+\Gamma(\vartheta_1,\lambda_1)=0$ due to
\eqref{eq30}.
\end{proof}
\begin{rem}
{\rm In fact, \eqref{eq13} admits a unique nonnegative solution
$(X(t),Y(t))_{t\geq0}$ under the weaker condition
$$
m(\vartheta_1,\lambda_1):=\vartheta_1^2\int_Uu^2\lambda_1(\d
u)<-4\beta_1, $$ rather than \eqref{eq30}, which is imposed just to
guarantee \eqref{eq29}. }
\end{rem}

For the two-factor model, we have the following result.

\begin{thm}\label{Theorem 3.1}
{\rm Under $(A5)-(A7)$ and \eqref{eq30},
\begin{equation*}
\lim\limits_{t\rightarrow\infty}\frac{1}{t^\mu}\int_0^tY(s)\d
s=\frac{\nu}{4\beta_1\beta_2}, \ \ \mbox{ a.s. }
\end{equation*}}
\end{thm}

\begin{proof}
By $(A5)$ and Theorem \ref{Theorem 1.3} we can deduce that
\begin{equation}\label{eq8}
\lim\limits_{t\rightarrow\infty}\frac{1}{t^\mu}\int_0^tX(s)\d
s=-\frac{\nu}{2\beta_1}, \ \ \mbox{ a.s. }
\end{equation}
On the other hand, for $\theta\in[1,2\mu]$ such that \eqref{eq2}, if
there exists $C>0$ such that
\begin{equation}\label{eq9}
\limsup_{t\rightarrow\infty}\frac{1}{t^{\theta}}\int_0^tX^2(s)\d
s\leq C, \ \ \mbox{ a.s., }
\end{equation}
which, together with \eqref{eq8} and Theorem \ref{Theorem 1.3},
leads to
\begin{equation*}
\lim\limits_{t\rightarrow\infty}\frac{1}{t^\mu}\int_0^tY(s)\d
s=\frac{\nu}{4\beta_1\beta_2}, \ \ \mbox{ a.s. }
\end{equation*}
Therefore, we only need  to verify 
\eqref{eq9}. By the It\^o formula and the Young inequality
\eqref{eq14}, it follows from \eqref{eq13} that
\begin{equation*}
\begin{split}
\d X^2(t) &=\Big\{(4\beta_1+m(\vartheta_1,\lambda_1))
X^2(t)+2\delta(t)X(t)+\sigma^2_1X(t)X^{2\gamma_1}(t-\tau)\Big\}\d
t\\
&\quad+2\sigma_1 X^{\frac{3}{2}}(t)X^{\gamma_1}(t-\tau)\d W_1(t)+\vartheta_1\int_U(2u+\vartheta_1u^2)X^2(t)\tilde{N}_1(\d u,\d t)\\
&\leq\Big\{(4\beta_1+\epsilon+m(\vartheta_1,\lambda_1))X^2(t)+\epsilon
X^2(t-\tau)+C(\epsilon)(1+\delta^2(t))\Big\}\d t\\
&\quad+2\sigma X^{\frac{3}{2}}(t)X^{\gamma_1}(t-\tau)\d
W_1(t)+\vartheta_1\int_U(2u+\vartheta_1u^2)X^2(t)\tilde{N}_1(\d u,\d
t)
\end{split}
\end{equation*}
for sufficiently small $\epsilon>0$ and some constant
$C(\epsilon)>0$. Integrating from $0$ to $t$ on both sides leads to
\begin{equation*}
\begin{split}
X^2(t)-\xi^2(0) &\leq \epsilon\|\xi\|^2\tau+(4\beta_1+2\epsilon
+m(\vartheta_1,\lambda_1))\int_0^tX^2(s)\d
s+C(\epsilon)\int_0^t(1+\delta^2(s)) \d
s\\
&\quad+2\sigma \int_0^tX^{\frac{3}{2}}(s)X^{\gamma_1}(s-\tau)\d
W_1(s)+\vartheta_1\int_0^t\int_U(2u+\vartheta_1u^2)X^2(s)\tilde{N}_1(\d
u,\d s).
\end{split}
\end{equation*}
By virtue of \eqref{eq30}, we can choose $\epsilon>0$ such that
$\tilde{\kappa}:=4\beta_1+2\epsilon +m(\vartheta_1,\lambda_1)<0$.
Thus  for $\theta\in[1,2\mu]$ such that $(A5)$
\begin{equation*}
\begin{split}
\frac{1}{t^{\theta}}\int_0^tX^2(s)\d s&\leq \frac{C(1+t)}{
t^{\theta}}+\frac{C}{\tilde{\kappa} t^{\theta}}\int_0^t\delta^2(s)\d
s\\
&\quad+\frac{2\sigma}{\tilde{\kappa} t^{\theta}}\int_0^t
X^{\frac{3}{2}}(s)X^{\gamma_1}(s-\tau)\d
W_1(s)+\frac{\vartheta_1}{\tilde{\kappa}
t^{\theta}}\int_0^t\int_U(2u+\vartheta_1u^2)X^2(s)\tilde{N}_1(\d
u,\d s).
\end{split}
\end{equation*}
By virtue of $\theta\in[1,2\mu]$ and \eqref{eq2}, note that the
first two terms on the right hand side are finite almost surely. 
In order to prove \eqref{eq9},  by Lemma
\ref{Lemma 1.2}  we only need to show that
\begin{equation*}
J_1(\infty):=\int_0^\infty\frac{X^{\frac{3}{2}}(s)X^{\gamma_1}(s-\tau)}{(1+s)^\theta}\d
W_1(s) \mbox{ and }
J_2(\infty):=\int_0^\infty\int_U\frac{X^2(s)}{(1+s)^\theta}\tilde{N}_1(\d
u,\d s)
\end{equation*}
exist a.s. For each $n>\|\xi\|$ define a stopping time
\begin{equation*}
\rho_n:=\inf\Big\{t\geq0\Big|\int_0^t\frac{\delta^4(s)}{(1+s)^{2\theta}}\d
s\geq n\Big\}.
\end{equation*}
By $(A5)$ it is easy to see that $\{\rho_n=\infty\}\uparrow\Omega$.
Following the argument of Theorem \ref{Theorem 1.3}, in what follows
we only need to show that
\begin{equation*}
M(t):=\int_0^t\frac{X^{\frac{3}{2}}(s)X^{\gamma_1}(s-\tau)}{(1+s)^\theta}1_{\{s\leq\rho_n\}}\d
W_1(s)
\end{equation*}
is $L_2$-bounded. By the It\^o isometry and the Young inequality
\eqref{eq14}, compute that
\begin{equation*}
\begin{split}
\mathbb{E}\left|M(t)\right|^2&=\mathbb{E}\int_0^t\frac{X^3(s)X^{2\gamma_1}(s-\tau)}{(1+s)^{2\theta}}1_{\{s\leq\rho_n\}}\d
s\\
&\leq
C+\mathbb{E}\int_0^t\frac{X^4(s)}{(1+s)^{2\theta}}1_{\{s\leq\rho_n\}}\d
s+\mathbb{E}\int_0^t\frac{X^4(s-\tau)}{(1+s)^{2\theta}}1_{\{s\leq\rho_n\}}\d
s\\
&:=C+J_1(t)+J_2(t).
\end{split}
\end{equation*}
For $\kappa>0$ by \eqref{eq29}, 
\begin{equation}
\begin{split}
J_1(t)&\leq\int_0^t\frac{e^{\kappa\beta_1 s}\mathbb{E}\{e^{-p\beta_1
(s\wedge\tau_n)}X^4(s\wedge\tau_n)\}}{(1+s)^{2\theta}}\d s\\
&\leq C\int_0^t\frac{e^{\kappa\beta_1
s}\Big\{1+s+\int_0^{s}e^{-\kappa\beta_1
r}\mathbb{E}(\delta^4(r)1_{\{r\leq\tau_n\}})\d
r\Big\}}{(1+s)^{2\theta}}\d s\\
&\leq C+C\int_0^t\frac{e^{\kappa\beta_1
s}}{(1+s)^{2\theta}}\int_0^{s}e^{-\kappa\beta_1
r}\mathbb{E}(\delta^4(r)1_{\{r\leq\tau_n\}})\d r\d s\\
&=C+C\int_0^te^{-\kappa\beta_1
r}\mathbb{E}(\delta^4(r)1_{\{r\leq\tau_n\}})\int_r^{t}\frac{e^{\kappa\beta_1
s}}{(1+s)^{2\theta}}\d s\d r\\
&\leq C(1+n).
\end{split}
\end{equation}
Similarly, we can get that $J_2(t)\leq C(1+n)$ and $J_2(\infty)$ exists. The proof is therefore
complete.
\end{proof}

\begin{rem}
{\rm By checking  the argument of Theorem \ref{Theorem 3.1}, it is
easy to see that Theorem \ref{Theorem 3.1} is still true for the
two-factor CIR-type mode \eqref{eq13}  with delay $\tau =0$ whenever
$\gamma_i\in[0,\frac{1}{2}),i=1,2$. On the other hand, the
 model \eqref{eq13} is not covered by
\cite[Theorem 2]{z09} due to the fact that the jump-diffusion
coefficient is Lipschitz continuous, but not H\"older continuous
with exponent $\frac{1}{2}$. }
\end{rem}

\begin{rem}
{\rm Stochastic models
under regime-switching have recently been developed to model various
financial quantities, e.g., option pricing, stock returns, and
portfolio optimisation. In particular, the CIR-type model under
regime-switching has found its considerable use as a model for
volatility and interest rate. Hence, it is also interesting to
discuss the long-term behavior of CIR-type model under
regime-switching
\begin{equation*}
\begin{cases}
\d X(t)=\{2\beta(r(t)) X(t) +\delta(t)\}\d t+\sigma(r(r))
|X(t)|^\theta\d W(t),\\
 X(0)=x \mbox{ and } r(0)=i_0,
\end{cases}
\end{equation*}
where $\theta\in[\frac{1}{2},1]$ and $r(t)$ is a continuous-time
Markov chain with a finite state space. This will  be presented
in a forthcoming paper. }
\end{rem}

\end{document}